\documentclass[oneside,11pt]{article}

\usepackage{amsmath}
\usepackage{amssymb}
\usepackage{pxfonts}
\usepackage{graphicx}
\usepackage{eucal}
\usepackage{mathrsfs}
\usepackage{theorem}
\usepackage{pifont}
\usepackage{sectsty}
\usepackage{amscd}
\usepackage{color}
\usepackage{fancyhdr}
\usepackage{framed}
\usepackage[all]{xy}
\usepackage[backref=page]{hyperref}
\usepackage[normalem]{ulem}
\usepackage{enumitem}
\usepackage{setspace}

\definecolor{shadecolor}{rgb}{0.8,0.8,0.8}

\theoremheaderfont{\fontfamily{pzc}\bfseries\large}
\newtheorem{theorem}{Theorem}[section]

\newtheorem{lemma}[theorem]{Lemma}
\newtheorem{proposition}[theorem]{Proposition}

\newtheorem{corollary}[theorem]{Corollary}

{\theorembodyfont{\rmfamily}

\newtheorem{example}{Example}[section]

\newcommand{\specexercise}[1]{}

\newenvironment{proof}{{\flushleft \emph{Proof:}}}{\hfill\ding{110}}
\newenvironment{proof1}[1]{{\flushleft \emph{Proof #1}}}{\hfill\ding{110}}

\newenvironment{remark}{{\flushleft \fontfamily{pzc}\bfseries\large Remark:}}{}

\newcommand{\dist}{\operatorname{dist}}
\newcommand{\SO}[1]{\text{SO}(#1)}
\newcommand{\id}{\operatorname{id}}
\newcommand{\diag}{\operatorname{diag}}
\newcommand{\tr}{\operatorname{tr}}

\newcommand{\argmin}{\operatornamewithlimits{arg\, min}}
\newcommand{\argmax}{\operatornamewithlimits{arg\, max}}

\newcommand{\brk}[1]{\left(#1\right)}          
\newcommand{\BRK}[1]{\left\{#1\right\}}        
\newcommand{\Abs}[1]{\left| #1 \right|}        
\newcommand{\inner}[1]{\left\langle#1\right\rangle}      

\newcommand{\beq}{\begin{equation}}
\newcommand{\eeq}{\end{equation}}

\newcommand{\Emph}[1]{{\slshape\bfseries #1}}  

\newcommand{\weakly}[1]{\stackrel{#1}{\rightharpoonup}}

\newcommand{\R}{\mathbb{R}}
\newcommand{\calR}{\mathcal{R}}
\newcommand{\calW}{\mathcal{W}}
\newcommand{\calQ}{\mathcal{Q}}

\newcommand{\calH}{\mathcal{H}}
\newcommand{\calP}{\mathcal{P}}
\newcommand{\bbM}{\mathbb{M}}
\newcommand{\bbZ}{\mathbb{Z}}
\newcommand{\e}{\varepsilon}
\newcommand{\W}{\Omega}
\newcommand{\w}{\omega}

\newcommand{\pl}{\partial}

\renewcommand{\thefootnote}{\fnsymbol{footnote}}
\numberwithin{equation}{section}

\begin{document}

\title{Reference configurations vs.~optimal rotations:\\ a derivation of linear elasticity from finite elasticity\\ for all traction forces
}
\author{Cy Maor and Maria Giovanna Mora\thanks{Corresponding author}}

\newcommand{\Addresses}{{
  \bigskip
  \footnotesize

  C.~Maor, \textsc{Einstein Institute of Mathematics,
  The Hebrew University of Jerusalem, Israel}\par\nopagebreak
  \textit{E-mail address}: \texttt{cy.maor@mail.huji.ac.il}

  \medskip

  M.G.~Mora, \textsc{Dipartimento di Matematica, Universit\`a di Pavia, Italy}\par\nopagebreak
  \textit{E-mail address}: \texttt{mariagiovanna.mora@unipv.it}

}}

\date{}
\maketitle
\renewcommand{\thefootnote}{\arabic{footnote}}

\abstract{
We rigorously derive linear elasticity as a low energy limit of pure traction nonlinear elasticity.
Unlike previous results, we do not impose any restrictive assumptions on the forces, and obtain a full $\Gamma$-convergence result.
The analysis relies on identifying the correct reference configuration to linearize about, and studying its relation to the rotations preferred by the forces ({\em optimal rotations}).
The $\Gamma$-limit is the standard linear elasticity model, plus a term that penalizes for fluctuations of the reference configurations from the optimal rotations. However, on minimizers this additional term is zero and the limit energy reduces to standard linear elasticity.
}

{\footnotesize 
\tableofcontents
}

\section{Introduction --- how to choose a reference configuration?}
\paragraph{Derivation of linear elasticity from finite elasticity}
In nonlinear (or finite) hyperelasticity, the elastic problem consists of minimizing an elastic energy over deformations $y:\W\to \R^n$, where $\W\subset \R^n$ is the elastic body.
\emph{Linear elasticity} is the linearization of this problem about a reference configuration: under the assumption that the displacement $u(x) := y(x)-x$ is small, one obtains a quadratic energy-minimization problem for $u$.
While this derivation of linear elasticity has been a textbook material for a very long time, only less than 20 years ago the first fully rigorous justification of it was obtained, via variational convergence, in \cite{DNP02}.
There, the authors considered the elastic energy of the type 
\[
\bar{J}_\e(y) := \int_\W \calW(x,\nabla y)\,dx - \e\int_\W g\cdot y\,dx, \qquad y\in W^{1,2}_{\e v_0}(\W;\R^n),
\]
where $W(x,A)$ is the elastic energy density, $g\in L^2(\W;\R^n)$ is the body forces, and $W^{1,2}_{\e v_0}(\W;\R^n)$ is the space of all maps $y\in W^{1,2}(\W;\R^n)$ such that $y = x + \e v_0$ on $\pl\W_D$, where $v_0$ is a given vector field
and $\pl\W_D$ is a prescribed subset of $\pl\W$.
They showed that the functionals $\frac{1}{\e^2}(\bar{J}_\e(y_\e) - \bar{J}_\e(\id))$, where $\id :\W\to \W$ is the identity map, $\Gamma$-converge to a linear elastic functional
\beq
\label{eq:linearized_Dirichlet}
I(u) = \int_\W \calQ(x,e(u))\,dx - \int_\W g\cdot u\,dx,
\eeq
where $u$ is the limit of the rescaled displacements $u_\e = \frac{1}{\e}(y_\e (x)-x)$, $e(u)$ is its symmetric gradient, and $\calQ$ is the quadratic form obtained from linearizing $\calW$ at the identity (see \eqref{eq:C_2_regularity}).
They also showed the associated compactness result; namely, if $\bar{J}_\e(y_\e) - \bar{J}_\e(\id)\le C\e^2$, then $u_\e$ weakly converge to some $u$ (modulo a subsequence).

Of course, the map $\id:\W\to \R^n$ is not the only reference configuration of the elastic body $\W$; any isometric embedding $Rx + c$, where $R\in \SO{n}$ and $c\in \R^n$, is.
Nevertheless, the choice of $\id$ as a reference configuration in \cite{DNP02} is a natural one, as they show that boundary conditions force $y-\id$ to be small in $W^{1,2}$.

A recent paper, \cite{MPT19}, approached the analogous problem, but with Neumann boundary conditions instead of Dirichlet. That is, they considered the pure traction problem
\beq\label{eq:energy_full}
\bar{J}_\e(y) := \int_\W \calW(x,\nabla y)\,dx -  \e\int_{\pl \W} f\cdot y\, \,d\calH^{n-1}- \e\int_\W g\cdot y\,dx, \qquad y\in W^{1,2}(\W;\R^n),
\eeq
where $f\in L^2(\pl\W;\R^n)$ and $g\in L^2(\W;\R^n)$ are the traction forces and body forces, respectively, which 
are \emph{equilibrated} in the sense that the energy $\bar{J}_\e$ is invariant to translations.
Furthermore, they assume a certain non-degeneracy condition (called \emph{compatibility} there); as explained later on, it is equivalent to the assumption that among all rigid motions, $\bar{J}_\e$ is minimized at $\id$, which is a unique minimizer (up to translations).
The fact that $\id$ is a minimizer among rigid motions can always be guaranteed by rotating the whole system; the fact that it is a unique minimizer, however, does limit the admissible forces. 

Under these assumptions, as in the Dirichlet case, they analyze the energy $J_\e(y):= \bar{J}_\e(y) - \bar{J}_\e(\id)$.
The analysis in this case turns out to be trickier than in the Dirichlet case, with some surprising results:
\begin{enumerate}
\item It turns out that a sequence of displacements $u_\e = \frac{1}{\e}(y_\e(x)-x)$ associated with approximate minimizers $y_\e$ of $\frac{1}{\e^2} \bar{J}_\e$ needs not to be bounded in $W^{1,2}$; 
	in fact, one can only obtain, after moving to a subsequence, that $e(u_\e)\weakly{} e(u)$, and $\sqrt{\e} \nabla u_\e \to W$ for some $u\in W^{1,2}$ and $W\in \mathbb{M}^{n\times n}_{\text{skew}}$ \cite[Theorem~2.2]{MPT19}.
\item The limiting $u$ does not minimize the expected linear elastic functional \eqref{eq:linearized_Dirichlet}, but rather the energy
	\[
	\tilde{I}(u) = \min_{W\in \mathbb{M}^{n\times n}_{\text{skew}}} \int_\W \calQ\brk{x,e(u) - \frac{1}{2}W^2}\,dx - \int_{\pl \W} f\cdot u\, \,d\calH^{n-1}- \int_\W g\cdot u\,dx.
	\]
	This energy is further investigated in a sequel paper, \cite{MPT19b}.
\item Unlike \cite{DNP02}, there is no full $\Gamma$-limit, but rather a statement about approximate minimizers. 
\end{enumerate}

\paragraph{Reference configurations and optimal rotations}
The above-mentioned works defined the displacement with respect to a reference configuration that is dictated by the problem; that is, by the boundary conditions or the forces.
In this work, we show that by choosing, for a given deformation, the rigid motion closest to it as its reference configuration, one can obtain stronger and more general results.
More precisely, we define the reference configuration of a deformation $y\in W^{1,2}(\W;\R^n)$ as the map $Rx+c$, $R\in \SO{n}$, $c\in \R^n$ that minimizes the displacement, that is\footnote{The fact that the minimum here is comparable with the elastic energy of $y$ is the content of the celebrated Friesecke-James-M\"uller rigidity theorem \cite[Theorem~3.1]{FJM02b}, which is the key technical tool for rigorously establishing limiting theorems for low-energy elastic systems.}
\beq\label{eq:reference_configuration}
Rx+c \in \argmin\big\{\|y(x) - (Qx+d)\|_{W^{1,2}} ~:~ Q\in \SO{n},\,d\in \R^n\big\}.
\eeq
In this case, one should distinguish between the reference configuration induced by a deformation $y$, and the preferred rotations of the forces, which we call optimal rotations.
Formally, for the energy \eqref{eq:energy_full}, we define the \Emph{set of optimal rotations} as
\[
\calR := \argmax_{R\in \SO{n}} \BRK{ F(R)},
\]
where $F\in \brk{\bbM^{n\times n}}^*$ is the linear functional defined by the forces, that is,
\beq\label{eq:def_F}
F(A) := \int_{\pl \W} f\cdot A x\, \,d\calH^{n-1} + \int_\W g\cdot A x\, dx.
\eeq
In this setting the correct normalization of the energy to consider is 
$$
J_\e(y):= \bar{J}_\e(y) + \e\int_{\pl \W} f\cdot \bar R x\, \,d\calH^{n-1} +\e \int_\W g\cdot \bar R x\, dx, \qquad
\bar R\in\calR,
$$ 
that is, the deviation of $\bar J_\e$ from its value on optimal rotations.
By rotating the system, we can always assume that $I\in \calR$ and thus, define $J_\e(y) := \bar{J}_\e(y) - \bar{J}_\e(\id)$.
As shown in Corollary~\ref{cor:compatibility}, the compatibility assumption of \cite{MPT19} is equivalent to saying that $\calR=\{I\}$.

\paragraph{Main results}
In this paper we address the pure traction elastic problem \eqref{eq:energy_full}, using the definitions of reference configurations, optimal rotations, and normalized energy as discussed above.
That is, for a given deformation $y_\e\in W^{1,2}(\W;\R^n)$, whose reference configuration according to \eqref{eq:reference_configuration} is $R_\e x+ c_\e$, we define its \Emph{rescaled displacement} by
\[
u_\e = \frac{1}{\e} R_\e^T \brk{y_\e - (R_\e x + c_\e)}.
\]
We obtain the following:
\begin{enumerate}
\item First, we prove that the set of optimal rotations $\calR$ is a totally-geodesic submanifold of $\SO{n}$ (Proposition~\ref{pn:calR_manifold}).
	This geometric observation is important for the following analysis. We also give a complete classification of the possible optimal rotations in dimensions $n=2,3$ (Section~\ref{sec:examples}).
\item Compactness (Theorem~\ref{thm:general_compactness}): 
	If $\frac{1}{\e^2}J_\e(y_\e)$ is bounded, then, modulo a subsequence, we have
	\begin{itemize}
	\item $u_\e \weakly{} u_0$ in $W^{1,2}(\W;\R^n)$,
	\item $R_\e \to R_0$ for some $R_0\in \calR$,
	\item $\frac{1}{\sqrt{\e}}(R_\e - \calP(R_\e)) \to A_0$, where $\calP(R_\e)$ is the projection of $R_\e$ onto $\calR$, and $A_0$ is an element of the normal bundle at $R_0$ of $\calR$ in $\SO{n}$.
		We can write $A_0 = R_0 W_0$ for some $W_0 \in \mathbb{M}^{n\times n}_{\text{skew}}$.
	\end{itemize} 
\item $\Gamma$-convergence (Theorem~\ref{thm:general_Gamma}): Under the above notion of convergence $y_\e \to (u_0,R_0,W_0)$, the functional $J_\e(y_\e)$ $\Gamma$-converges to
	\[
	I(u_0,R_0,W_0) := \int_\W \calQ(x,e(u_0(x)))\,dx - \int_{\pl \W} f\cdot R_0 u_0 \, \,d\calH^{n-1} -  \int_\W g\cdot R_0 u_0 \, dx - \frac{1}{2}F(R_0W_0^2),
	\]
	where $F$ is defined in \eqref{eq:def_F}.\footnote{Under the assumption that $\calR=\{I\}$, this functional coincides with the functional obtained in \cite{MPT19}, under the change $u_0(x) \mapsto u_0(x) - \frac{1}{2} W_0^2 x$ in the functional above.}
\end{enumerate}

It turns out that this viewpoint, compared to the one of \cite{MPT19}, provides better compactness properties, a full $\Gamma$-convergence result, and it is valid for \emph{all} equilibrated forces (in particular, the assumption $\calR=\{I\}$ is not necessary for a rigorous validation of linear elasticity).
On a more technical point, our proofs are simpler, and work for any dimension $n$, whereas the proofs in \cite{MPT19} rely on the Rodrigues rotation formula (see \eqref{eq:Rodrigues}), which is only valid for $n=2,3$.

Our approach also gives a geometric interpretation to the difference between the Dirichlet and Neumann derivations of linear elasticity: 
whereas in the Dirichlet case, the rotational part $R_\e$ of the reference configuration differs from the rotation prescribed by the boundary data by an order of $\e$ (see \cite{DNP02}, equation (3.14)), 
in the Neumann case the distance between $R_\e$ and the optimal rotations prescribed by the forces is only of order $\sqrt{\e}$.\footnote{We note that a related observation appears in \cite[Remark 2.9]{MPT19}.}
From a mechanical point of view, it means that a low energy pure traction elastic body can fluctuate more compared to a low energy elastic body which is clamped in part of its boundary.

Finally, we note that the term $-\frac{1}{2}F(R_0W_0^2)$ that appears in the limiting energy, does not appear in the standard linear elastic energy, such as \eqref{eq:linearized_Dirichlet} (this can be viewed as a manifestation of the ``gap'', as it is called in \cite{MPT19}, between standard linear elasticity and its rigorous derivation from finite elasticity for pure traction problems).
This term represents the elastic cost of fluctuations of the reference configurations from the optimal rotations; in the Dirichlet case, these fluctuations are smaller, and their elastic cost does not appear in this energy scaling.
However, note that the term $-\frac{1}{2}F(R_0W_0^2)$ is non-negative, since $R_0$ is an optimal rotation (see \eqref{eq:forces_EL_new} below); therefore, from a minimization point of view, we can always choose $W_0=0$, thus eliminating it.
More precisely, we show that minimizers of $J_\e$ converge to minimizers of $I$ of the form $(u_0,R_0,0)$, which reduces $I$ to the standard linear elasticity energy (see Theorem~\ref{thm:conv_min}), with the slight difference that formal derivations of linear elasticity typically focus on linearization about a fixed optimal rotation and thus do not consider $R_0$ explicitly. In other words, the standard linear elasticity energy gives the correct asymptotic description of minimizers of finite elasticity for small forces not only in the Dirichlet case, but also for all pure traction problems.

After this work was essentially complete, we learned about the papers \cite{MP20b} and \cite{JS20}, where the authors study the derivation of pure traction linear elasticity from finite elasticity for incompressible materials. In \cite{MP20b} the external forces are assumed to satisfy the same compatibility condition as in \cite{MPT19}, that is, in our language $\calR=\{I\}$. In \cite{JS20} the assumptions on the forces imply the other extreme, namely that $\calR=\SO{n}$. We believe that our approach, adapted to the incompressible case, should be able to unify these two results and extend them to all forces.

\paragraph{Structure of this paper}
In Section~\ref{sec:model} we describe in more detail the elastic energy $J_\e$ that we are considering, and define the set of optimal rotations $\calR$ induced by it.
In Section~\ref{sec:prel_est} we give some standard preliminary estimates, regarding (a) the distance between a deformation and its reference configuration (Lemma~\ref{lem:using_FJM}, in which the Friesecke-James-M\"uller rigidity theorem comes into play), and (b) the scaling of the infimum of elastic energy $J_\e$ (Proposition~\ref{prop:inf_energy}), which justifies the energy scaling considered.
In Section~\ref{sec:geometry_calR} we treat the geometry of the set of optimal rotations $\calR$, and show that it is a totally-geodesic submanifold of $\SO{n}$ (Proposition~\ref{pn:calR_manifold}).
In Section~\ref{sec:main_results} we state and prove our main results --- compactness (Theorem~\ref{thm:general_compactness}), $\Gamma$-convergence (Theorem~\ref{thm:general_Gamma}), and convergence of minimizers (Theorem~\ref{thm:conv_min}).
In Section~\ref{sec:examples} we give a full classification of the possible sets of optimal rotations that can arise in two and three dimensions, and provide examples for each.

\section{The model}\label{sec:model}

Let $\W \subset \R^n$ be a Lipschitz domain, and consider the energy $\bar{J}_\e: W^{1,2}(\W;\R^n) \to \R\cup \{+\infty\}$, defined by
\[
\bar{J}_\e(y) := \int_\W \calW(x,\nabla y) \,dx - \e\int_{\pl \W} f\cdot y\, \,d\calH^{n-1} - \e \int_\W g\cdot y\, dx,
\]
where $\calW: \W\times \mathbb{M}^{n\times n} \to [0,\infty]$ is the elastic energy density, a Carath\'eodory function satisfying the following assumptions:
\begin{enumerate}[label=(\alph*)]
\item \label{itm:frame_ind} Frame indifference: $\calW(x,RA) = \calW(x,A)$ for a.e.\ $x\in \W$, all $A\in \mathbb{M}^{n\times n}$ and $R\in \SO{n}$. 
\item \label{itm:well} $\calW(x,A) = 0$ if and only if $A\in \SO{n}$. 
\item \label{itm:coercive} Coercivity: There exists $c>0$ such that $\calW(x,A) \ge c \dist^2(A,\SO{n})$ for all $A\in \mathbb{M}^{n\times n}$ and a.e.\ $x\in \W$.
\item \label{item:regularity} Regularity: There exists a neighborhood of $\SO{n}$ in which $\calW(x,\cdot)$ is $C^2$ uniformly in $x$:
	\beq\label{eq:C_2_regularity}
	\left|\calW(x, I + B) - \calQ(x,B)\right| \le \omega(|B|), \qquad \calQ(x,B) := \frac{1}{2}D_A^2\calW(x,I)(B,B)
	\eeq
	where $\omega:[0,\infty)\to [0,\infty]$ is a function satisfying $\lim_{t\to 0} \omega(t)/t^2 = 0$. 
	Moreover, $D_A^2\calW(\cdot,I)$ is a bounded function in $\W$.
\end{enumerate}
We note that assumptions \ref{itm:well} and \ref{itm:coercive} imply that 
\beq\label{eq:Q_symmetric}
\calQ(x,B) = \calQ\brk{x,{\rm sym}\, B}\geq c\,|{\rm sym}\, B|^2
\eeq
for all $B\in \mathbb{M}^{n\times n}$ and a.e.\ $x\in \W$.

We assume that the forces $f$ and $g$ are \emph{equilibrated}, that is,
\beq
\label{eq:total_force_2}
\int_{\pl \W} f \, \,d\calH^{n-1} + \int_\W g\, dx = 0.
\eeq
Without this assumption, by changing $y\mapsto y+c$ we can make $\bar{J}_\e$ arbitrary small, i.e., $\inf \bar{J}_\e = -\infty$.

Let
\[
F\in \brk{\bbM^{n\times n}}^*, \qquad F(A) := \int_{\pl \W} f\cdot A x\, \,d\calH^{n-1} + \int_\W g\cdot A x\, dx,
\]
and define the \textbf{set of optimal rotations} $\calR$ by
\[
\calR := \argmax_{R\in \SO{n}} \BRK{ F(R)}.
\]
Fix $\bar{R}\in \calR$. By changing $f\mapsto \bar{R}^T f$, $g\mapsto \bar{R}^T g$ and $y\mapsto \bar{R}^T y$, we can assume without loss of generality that $\bar{R} = I$.
In particular, we have
\beq\label{eq:forces_new}
F(R-I) = \int_{\pl \W} f\cdot (R-I) x\, \,d\calH^{n-1} +  \int_\W g\cdot (R-I) x\, dx \le 0,
\eeq
with equality holding if and only if $R\in \calR$.

Let $I_\e$ be the elastic part of $\bar{J}_\e$, i.e.,
\[
I_\e(y) := \int_\W \calW(x,\nabla y) \,dx,
\]
and denote
\[
\begin{split}
J_\e(y) &:= \bar{J}_\e(y) - \bar{J}_\e(\id)\\
	&= I_\e(y) - \e\int_{\pl \W} f\cdot (y-x)\, \,d\calH^{n-1} - \e \int_\W g\cdot (y-x)\, dx. 
\end{split}
\]

\section{Preliminary estimates}\label{sec:prel_est}

We begin with some preliminary calculations: In Lemma~\ref{lem:using_FJM} we show that if $J_\e(y_\e)\le C\e^2$, then the $W^{1,2}$-distance between $y_\e$ and its reference configuration is of order $\e$.
In Proposition~\ref{prop:inf_energy} we show that
\[
-C\e^2 \le \inf_{W^{1,2}} J_\e \le 0,
\] 
for some $C>0$ depending on the forces $f,g$ and the energy density $W$.
These motivate the study of the $\Gamma$-limit of $\frac{1}{\e^2} J_\e$.

In this section, we use the notation $A_\e\lesssim B_\e$ if $A_\e \le C B_\e$ for some constant $C>0$ that is independent of $\e$, but can depend on $\W$, the constant $c$ in the coercivity assumption \ref{itm:coercive}, and other fixed quantities.

\begin{lemma}
\label{lem:using_FJM}
If $J_\e(y_\e) \leq C\e^2$, then $I_\e(y_\e) = O(\e^2)$ and there exist a sequence $R_\e\in \SO{n}$ and constants $c_\e\in \R^n$ such that 
\[
\|y_\e - (R_\e x +c_\e) \|_{W^{1,2}}\lesssim \e.
\]
If $R_\e'\in \SO{n}$ is another sequence with respect to which this holds, then $|R_\e - R'_\e| \lesssim \e$.
\end{lemma}

\begin{remark}
As we will show later, the fact that $|R_\e - R'_\e| \lesssim \e$ implies that we can regard any sequence $R_\e x+c_\e$ for which this lemma holds as reference configurations of the sequence $y_\e$, without changing the results of this paper.
\end{remark}

\begin{proof}
By the Friesecke-James-M\"uller rigidity theorem \cite[Theorem~3.1]{FJM02b}, the coercivity assumption \ref{itm:coercive} on $\calW$ implies that there exist $R_\e\in \SO{n}$ such that
\[
\|\nabla y_\e - R_\e\|_{L^2} \lesssim \brk{I_\e(y_\e)}^{1/2}.
\]
This also implies that, for an appropriate constant $c_\e$,
\[
\|Y_\e \|_{W^{1,2}}\lesssim \brk{I_\e(y_\e)}^{1/2},
\]
where $Y_\e:= y_\e - R_\e x - c_\e$. 
From the trace theorem, a similar bound also holds for $L^2$-norm of the trace of $Y_\e$.
Therefore, we only need to prove that $I_\e(y_\e) = O(\e^2)$.
Using the inequalities above, \eqref{eq:total_force_2} and \eqref{eq:forces_new}, we have
\[
\begin{split}
I_\e(y_\e) &= J_\e(y_\e) + \e\int_{\pl \W} f\cdot (y_\e-x)\, \,d\calH^{n-1} + \e \int_\W g\cdot (y_\e-x)\, dx \\
	&\le C\e^2 + \e\int_{\pl \W} f\cdot (y_\e-x)\, \,d\calH^{n-1} + \e \int_\W g\cdot (y_\e-x)\, dx \\
	&=C\e^2 + \e\int_{\pl \W} f\cdot Y_\e\, \,d\calH^{n-1} + \e \int_\W g\cdot Y_\e \, dx \\
		&\qquad\;\;\, + \e\int_{\pl \W} f\cdot (R_\e-I)x\, \,d\calH^{n-1} + \e \int_\W g\cdot (R_\e-I)x\, dx \\
	&\le C\e^2 + \e\int_{\pl \W} f\cdot Y_\e \, \,d\calH^{n-1} + \e \int_\W g\cdot Y_\e \, dx \\
	&\le C\e^2 + \e\|f\|_{L^2(\pl\W)} \|Y_\e\|_{L^2(\pl\W)} + \e \|g\|_{L^2(\W)} \|Y_\e\|_{L^2(\W)} \\
	&\lesssim \e^2 + \e\brk{\|f\|_{L^2(\pl\W)} + \|g\|_{L^2(\W)}} \brk{I_\e(y_\e)}^{1/2} \\
	&\le \e^2 + \frac{\e^2}{\delta^2}\brk{\|f\|_{L^2(\pl\W)}^2 + \|g\|_{L^2(\W)}^2} + \delta^2 I_\e(y_\e), \\
\end{split}
\]
which completes the proof by choosing $\delta$ small enough.

Finally, the last statement follows since
\[
|R_\e - R_\e'| \lesssim \|\nabla y_\e - R_\e\|_{L^2} + \|\nabla y_\e - R_\e'\|_{L^2} \lesssim \e.
\]
\end{proof}

\begin{proposition}\label{prop:inf_energy}
There exists $C>0$ such that
\[
-C\e^2 \le \inf J_\e \le 0.
\]
\end{proposition}

\begin{proof}
The upper bound follows since $J_\e(\id) = 0$.
For the lower bound, consider a sequence of approximate minimizers $y_\e$, that is
\[
J_\e(y_\e) - \inf J_\e \le C'\e^2,
\]
for some $C'>0$.
In particular, $J_\e(y_\e) \le C'\e^2$, hence the results of Lemma~\ref{lem:using_FJM} hold.
We therefore have
\[
\begin{split}
J_\e(y_\e) &\ge -\e\int_{\pl \W} f\cdot (y_\e-x)\, \,d\calH^{n-1} - \e \int_\W g\cdot (y_\e-x)\, dx \\
	&\ge -\e\int_{\pl \W} f\cdot Y_\e \, \,d\calH^{n-1} - \e \int_\W g\cdot Y_\e \, dx - \e\int_{\pl \W} f\cdot (R_\e-I)x\, \,d\calH^{n-1} - \e \int_\W g\cdot (R_\e-I)x\, dx \\
	&\ge -\e\int_{\pl \W} f\cdot Y_\e \, \,d\calH^{n-1} - \e \int_\W g\cdot Y_\e \, dx \\
	&\ge -\e\|f\|_{L^2(\pl\W)} \|Y_\e \|_{L^2(\pl\W)} - \e \|g\|_{L^2(\W)}\|Y_\e \|_{L^2(\W)} \ge -C\e^2,
\end{split}
\]
for some constant $C>0$.
\end{proof}

\section{Geometry of the set of optimal rotations $\calR$}\label{sec:geometry_calR}

We recall that the tangent space to $\SO{n}$ at the identity is the space of skew-symmetric matrices, and at $R\in\SO{n}$ it is 
$\{RW ~:~ W\in \mathbb{M}^{n\times n}_{\text{skew}}\}$.
Moreover, for a fixed $R$, we have that $\SO{n} = \{Re^W ~:~ W\in \mathbb{M}^{n\times n}_{\text{skew}}\}$, and for every $R'\in \SO{n}$, there exists $W\in \mathbb{M}^{n\times n}_{\text{skew}}$ such that $R' = Re^W$ and the map $t\in[0,1]\mapsto Re^{tW}$ is a minimizing geodesic in $\SO{n}$ connecting $R$ and $R'$.

Let now $R\in \calR$ and $W\in \mathbb{M}^{n\times n}_{\text{skew}}$. From the definition of $\calR$ the function $\phi(t):=F(Re^{tW})$ satisfies $\phi'(0)=0$ and $\phi''(0)\leq0$. Thus, we deduce that
\beq\label{eq:forces_EL_new}
F(RW) = 0, \qquad
F(RW^2)  \le 0,
\eeq
for every $W\in \mathbb{M}^{n\times n}_{\text{skew}}$ and $R\in \calR$. We note that the first equation in \eqref{eq:forces_EL_new} for $R=I$, together with \eqref{eq:total_force_2}, provides the usual balance condition in linearized elasticity:
$$
\int_{\pl \W} f\cdot (W x+c) \, \,d\calH^{n-1} + \int_\W g\cdot (W x+c)\, dx = 0
$$
for every $W\in \mathbb{M}^{n\times n}_{\text{skew}}$ and $c\in\R^n$.

Our main result of this section is the following characterization of the set of optimal rotations:
\begin{proposition}\label{pn:calR_manifold}
$\calR$ is a closed, connected, boundryless, totally-geodesic submanifold of $\SO{n}$, and the tangent space of $\calR$ at $R_0$ is
\beq\label{eq:TR_R_0}
T\calR_{R_0} = \BRK{R_0 W ~:~ W\in \mathbb{M}^{n\times n}_{\text{skew}}, \quad F(R_0 W^2) = 0}.
\eeq
In particular, $T\calR_{R_0}$ is a linear space.
\end{proposition}

Recall that a totally-geodesic submanifold $\mathcal{M}$ of a manifold $\mathcal{N}$ is a submanifold, such that a length-minimizing curve in $\mathcal{M}$ between any two elements in $\mathcal{M}$ is also a length-minimizing curve in $\mathcal{N}$ (e.g., a hyperplane in Euclidean space).

\begin{corollary}\label{cor:compatibility}
An immediate corollary is that strict inequality in \eqref{eq:forces_EL_new} is equivalent to saying that $\calR$ is a singleton, i.e., $\calR=\{I\}$.
This strict inequality is the \emph{compatibility} assumption on the forces in \cite{MPT19} (see (2.25) there).
\end{corollary}

Proposition~\ref{pn:calR_manifold} is what we need for the compactness and $\Gamma$-convergence results.
Later on, in Section~\ref{sec:examples}, we give more details on the structure of $\calR$; 
in particular, we show that the second fundamental form of $\SO{n}$ in $\mathbb{M}^{n\times n}$ in the direction $F$ is negative semi-definite, and that the number of its zero principal curvatures corresponds to the dimension of $\calR$. 
This yields a complete classification of the possible optimal rotations in two and three dimensions.

We will prove Proposition~\ref{pn:calR_manifold} at the end of the section, after a few preliminaries.
For later use, we denote
\beq\label{eq:normal_space}
N \calR_{R_0} := \BRK{W\in \mathbb{M}^{n\times n}_{\text{skew}} ~:~ R_0 W \perp T\calR_{R_0}}.
\eeq
Note that $R_0 N\calR_{R_0}$ is the normal space of $T\calR_{R_0}$ in $T_{R_0} \SO{n}$.
Also, we define the projection operator
\beq\label{eq:projection}
\calP: \SO{n} \to \calR, \qquad \calP(Q) := \argmin \BRK{\dist_{\SO{n}}(Q,R) : R\in \calR }.
\eeq
Since $\calR$ is a closed submanifold, $\calP$ is well-defined in a neighborhood of $\calR$.
Here, $\dist_{\SO{n}}$ is the intrinsic distance in the manifold $\SO{n}$; that is,
\[
\dist_{\SO{n}}(Q,R) = \min \BRK{ |W| ~:~ W\in \mathbb{M}^{n\times n}_{\text{skew}}, Q = R e^W }.
\]
Note that this distance is equivalent to the regular (Frobenius) distance in $\mathbb{M}^{n\times n}$ (since $\SO{n}$ is a compact submanifold), and moreover,
\beq\label{eq:equiv_dist}
\dist_{\SO{n}}(Q,R) = |Q-R| + O(|Q-R|^2).
\eeq

Towards the proof of Proposition~\ref{pn:calR_manifold}, we start by recalling a few linear algebra facts:
any $W\in \mathbb{M}^{n\times n}_{\text{skew}}$ can be written as $R^T\Sigma R$, where $R\in \SO{n}$ and 
\beq\label{eq:skew_canonical_form}
\Sigma = \diag\brk{A(\lambda_1), A(\lambda_2),\ldots,A(\lambda_k),0,\ldots,0}, \quad 
A(\lambda) =\brk{\begin{matrix} 0 & \lambda \\ -\lambda & 0 \end{matrix} }, \quad \lambda_i \in \R\setminus \{0\}.
\eeq
From this, we have the following:
\begin{lemma}\label{lem:geodesics_SO}
Given a rotation $R\in \SO{n}$, any rotation $R'\in \SO{n}$ can be written as $R' = Re^W$, where $W\in \mathbb{M}^{n\times n}_{\text{skew}}$ and the values $\lambda_1,\ldots,\lambda_k$ in the representation \eqref{eq:skew_canonical_form} of $W$ belong to the interval $(-\pi,\pi]$.
\end{lemma}

\begin{proof}
We prove for the case $R=I$, that is,
that for each $W\in \mathbb{M}^{n\times n}_{\text{skew}}$ there exists $W'\in \mathbb{M}^{n\times n}_{\text{skew}}$ such that $e^W=e^{W'}$, and whose non-zero eigenvalues $\{\pm \lambda_i i\}_{i=1}^k$ satisfy $\lambda_i \in (-\pi,\pi]$.
For a general $R$ the result follow by multiplying everything from the left by $R$.
First, note that \eqref{eq:skew_canonical_form} implies that 
\beq\label{eq:exp_W}
\begin{split}
\cosh(W) &= I + \sum_{i=1}^k (\cos(\lambda_i)-1) R^T D_i R, \\
\sinh(W) &= \sum_{i=1}^k \sin(\lambda_i) R^T E_i R,
\end{split}
\eeq
where $\lambda_i$ and $R\in \SO{n}$ are as in \eqref{eq:skew_canonical_form}, and 
\[
(D_i)_{\alpha\beta} =
\begin{cases}
1 & \alpha=\beta=2i-1,2i, \\
0 & \text{otherwise.}
\end{cases}
\qquad
(E_i)_{\alpha\beta} =
\begin{cases}
1 	& \alpha=2i-1,\,\beta=2i, \\
-1	& \alpha=2i,\,\beta=2i-1,  \\
0 	& \text{otherwise.}
\end{cases}
\]
We note that $e^W = \cosh(W) + \sinh(W)$, and this is exactly the decomposition of $e^W$ into a symmetric ($\cosh(W)$) and a skew-symmetric ($\sinh(W)$) matrices.
Formulae \eqref{eq:exp_W} imply, in particular, that if 
\[
W' = R^T\diag\brk{A(\lambda_1'), A(\lambda_2'),\ldots,A(\lambda_k'),0,\ldots,0}R,
\]
where $\lambda_i' - \lambda_i \in 2\pi\bbZ$ for every $i$, then $e^W = e^{W'}$.
Thus it is possible to choose the $\lambda_i$s in any interval of length $2\pi$.
This completes the proof.
\end{proof}

Next, we note that for every $R_0\in\calR$, $R\in \SO{n}$ and $i$,
\beq\label{eq:FDi}
F(R_0R^T D_i R) \ge 0.
\eeq
Assume otherwise; without loss of generality, assume that
\[
F(R_0R^T D_1 R) = a < 0.
\]
Now, consider the matrix
\[
W = R^T \diag\brk{A(1), 0,\ldots,0} R \in \mathbb{M}^{n\times n}_{\text{skew}}.
\]
We have
\[
\cosh(tW) = I + (\cos(t)-1) R^T D_1 R,
\]
hence, for every $t\in (0,2\pi)$, using that $\sinh(tW)\in \mathbb{M}^{n\times n}_{\text{skew}}$ and thus, $F(R_0\sinh(tW))=0$ by \eqref{eq:forces_EL_new},
\[
F(R_0 e^{tW}) = F(R_0) + a(\cos(t)-1) > F(R_0),
\]
which is a contradiction to $R_0\in\calR$.

Now we can easily prove the following two Lemmas, that are the main building blocks towards Proposition~\ref{pn:calR_manifold}.
Lemma~\ref{lem:F_W_squared_zero} states that for any $W\in T\calR_{R_0}$ (see \eqref{eq:TR_R_0}), the whole $\SO{n}$-geodesic emanating from $R_0$ in direction $W$ belongs to $\calR$;
Lemma~\ref{lem:calR_geodesics} states that for any two elements $R_0,R_1\in \calR$, there exists a geodesic between them that belongs to $\calR$.

\begin{lemma}\label{lem:F_W_squared_zero}
If $R_0\in \calR$ and $W\in \mathbb{M}^{n\times n}_{\text{skew}}$ such that $F(R_0W^2) = 0$, then $R_0e^{tW} \in \calR$ for any $t\in \R$.
\end{lemma}

\begin{proof}
Let $W\in \mathbb{M}^{n\times n}_{\text{skew}}$ be such that $F(R_0W^2) = 0$.
Let us write $W$ in its canonical form \eqref{eq:skew_canonical_form}, with $\lambda_i\ne 0$.
Note that
\[
0=F(R_0W^2) = -\sum_{i=1}^k \lambda_{i}^2 F(R_0 R^T D_i R).
\]
By \eqref{eq:FDi} it follows that $F(R_0 R^T D_i R)=0$ for $i=1,\ldots,k$.
We then have
\[
F(R_0e^{tW}) = F(R_0) + \sum_{i=1}^k (\cos(\lambda_i t) - 1) F(R_0 R^T D_i R) = F(R_0),
\]
hence $R_0e^{tW}\in \calR$ for every $t\in \R$.
\end{proof}

\begin{lemma}\label{lem:calR_geodesics}
If $R_0,R_1$ are two distinct elements in $\calR$, then $\calR$ contains a geodesic of $\SO{n}$ that connects $R_0$ and $R_1$.
More precisely, if $R_1 = R_0 e^{W}$, where $W$ is of the form of Lemma~\ref{lem:geodesics_SO}, then
\[
\BRK{R_0e^{tW} ~:~ t\in \R} \subset \calR.
\]
\end{lemma}

\begin{remark}
In dimensions $n=2,3$, we can actually obtain that any geodesic between $R_0$ and $R_1$ lies in $\calR$; for $n>3$, this is no longer the case due to conjugate points. See Appendix~\ref{app:example} for details. 
\end{remark}

\begin{proof}
Let $R_0,R_1\in \calR$, and pick $W\in \mathbb{M}^{n\times n}_{\text{skew}}$ such that $R_1=R_0 e^{W}$, with $W$ of the form of Lemma~\ref{lem:geodesics_SO}.
We therefore have, for some $R\in \SO{n}$, that
\[
0 = F(R_1-R_0) = \sum_{i=1}^k a_i (\cos(\lambda_i) - 1),\qquad a_i = F(R_0 R^T D_i R)\ge 0,
\]
where we used \eqref{eq:FDi}.
Since $\lambda_i \in (-\pi,\pi]\setminus \{0\}$, it follows that $a_i=0$ for all $i$.
But then, for every $t\in \R$,
\[
F(R_0e^{tW} - R_0) = \sum_{i=1}^k a_i (\cos(t\lambda_i) - 1) = 0,
\]
hence $R_0e^{tW}\in \calR$ for every $t\in \R$.
\end{proof}

Finally, we prove Proposition~\ref{pn:calR_manifold}.
\begin{proof1}{of Proposition~\ref{pn:calR_manifold}:}
We first prove that the set
\[
T:=\BRK{W\in \mathbb{M}^{n\times n}_{\text{skew}} ~:~ \quad F(W^2) = 0}
\]
is a vector space.
It is obvious that $T$ is closed under scalar multiplication; the idea is to "zoom in" at the origin, where we can effectively treat the geodesics that connect two matrices in $\SO{n}$ as straight lines in the linear space of matrices:
Assume that $W_1,W_2\in T$;
Lemma~\ref{lem:F_W_squared_zero} implies that $e^{taW_1},e^{tbW_2}\in \calR$ for every $a,b\in \R$ and $t> 0$.
We will show that for small $t$, the midpoint of the geodesic between $e^{taW_1}$ and $e^{tbW_2}$ belongs to $\calR$, and that this midpoint is $\exp\brk{\frac{t}{2}(aW_1 + bW_2 + O(t))}$.
The previous lemmata will then imply that $\frac{t}{2}(aW_1 + bW_2 + O(t))\in T$; we will then ``zoom out'' and obtain that $aW_1 + bW_2 \in T$.
Indeed, consider, for small $t$, the geodesic between $e^{taW_1}$ and $e^{tbW_2}$.
We can write it as
\[
\tau\mapsto e^{taW_1}e^{\tau Z},
\]
where $e^Z = e^{-taW_1}e^{tbW_2}$, hence
\[
Z = tbW_2 - taW_1 + O(t^2).
\]
Since $|Z|=O(t)$, we obtain that for small enough $t$, all the eigenvalues of $Z$ are close to zero, hence Lemma~\ref{lem:calR_geodesics} implies that this geodesic belongs to $\calR$.
In particular, we have that the midpoint of this geodesic, $e^{taW_1} e^{Z/2}$, belongs to $\calR$; we can write it as
\[
e^{taW_1} e^{Z/2} = e^{Z'}, \qquad Z' = \frac{t}{2}\brk{aW_1 + bW_2} + O(t^2).
\]
Using Lemma~\ref{lem:calR_geodesics} again, we have that $e^{\tau Z'}\in \calR$ for every $\tau$, from which we obtain that $Z'\in T$.
Since $T$ is closed to scalar multiplication, we have that $2Z'/t\in T$, thus
\[
aW_1 + bW_2 + O(t) \in T,
\]
for every $t>0$, and since $T$ is a closed set, we have that $aW_1 + bW_2 \in T$.

We now claim that at the vicinity of $I$, $\calR$ is the image of the exponential map restricted to $T$.
Indeed, Lemma~\ref{lem:F_W_squared_zero} implies that the image of the exponential map, restricted to $T$, is in $\calR$.
On the other hand, Lemma~\ref{lem:calR_geodesics} implies that if $R\in \calR$ then $R = e^{W}$ for some $W\in T$.
This tells us that at the vicinity of $I$, $\calR$ is a manifold whose tangent space is $T$.

However, we can do this analysis around any $R_0\in \calR$, and thus $\calR$ is indeed a manifold whose tangent space is $T\calR_{R_0}$. 
Lemma~\ref{lem:calR_geodesics} implies that it is connected.
Since for each $R_0$, $\calR$ is locally homeomorphic to an open neighborhood of the zero element of the vector space $T\calR_{R_0}$, we have that $\calR$ has no boundary; 
since, by definition, $\calR$ is a set of maximizers of a continuous function, it is closed.
We therefore deduce that $\calR$ is a closed manifold.

Finally, Lemma~\ref{lem:F_W_squared_zero} implies that for any $W\in T_{R_0}\calR$, the $\SO{n}$-geodesic $R_0e^{tW}$ stays on the submanifold $\calR$, hence $\calR$ is totally geodesic.
\end{proof1}

\section{Main results}\label{sec:main_results}

\begin{theorem}[Compactness]\label{thm:general_compactness}
Let $y_\e \in W^{1,2}(\W;\R^n)$ be such that $J_\e(y_\e) \le C\e^2$, and let $R_\e x + c_\e$ be a reference configuration of $y_\e$, satisfying the results of Lemma~\ref{lem:using_FJM}. 
Denote the \Emph{rescaled displacement} of $y_\e$ by
\beq\label{eq:displacement}
u_\e(x) = \frac{1}{\e} R_\e^T\brk{y_\e(x) -(R_\e x + c_\e)}.
\eeq
We then have the following, up to moving to a subsequence:
\begin{itemize}
\item $u_\e \weakly{} u_0$ in $W^{1,2}(\W;\R^n)$,
\item $R_\e \to R_0\in \calR$,
\item $\frac{1}{\sqrt{\e}}\brk{R_\e - \calP(R_\e)} \to R_0W_0$, for some $W_0 \in N\calR_{R_0}$,
\end{itemize}
where $N\calR_{R_0}$ and $\calP$ were defined in \eqref{eq:normal_space}--\eqref{eq:projection}.
Moreover, we have that $R_0$, $W_0$ are independent of the choice of $R_\e$, and $u_0$ is independent up to a change by an infinitesimal isometry $Ax + b$, where $A\in \mathbb{M}^{n\times n}_{\text{skew}}$ and $b\in \R^n$.
\end{theorem}

\begin{theorem}[$\Gamma$-convergence] \label{thm:general_Gamma}
Under the convergence $y_\e \to (u_0,R_0,W_0)$ as defined in Theorem~\ref{thm:general_compactness}, we have
\[
\Gamma-\lim \frac{1}{\e^2} J_\e(y_\e) 
= \int_\W \calQ(x,e(u_0(x)))\,dx - \int_{\pl \W} f\cdot R_0 u_0 \, \,d\calH^{n-1} -  \int_\W g\cdot R_0 u_0 \, dx - \frac{1}{2}F(R_0W_0^2),
\]
where $\calQ$ is defined in \eqref{eq:C_2_regularity}.
In particular, this means
\begin{enumerate}
\item \textbf{Lower bound:} 
If $y_\e \to (u_0,R_0,W_0)$, then
\[
\liminf \frac{1}{\e^2} J_\e(y_\e) \ge \int_\W \calQ(x,e(u_0(x)))\,dx - \int_{\pl \W} f\cdot R_0 u_0 \, \,d\calH^{n-1} -  \int_\W g\cdot R_0 u_0 \, dx - \frac{1}{2}F(R_0W_0^2).
\]

\item  \textbf{Upper bound:} 
For every $u_0\in W^{1,2}(\W;\R^n)$, $R_0\in \calR$ and $W_0\in N\calR_{R_0}$, there exists $y_\e\in W^{1,2}(\W;\R^n)$ such that $y_\e\to (u_0,R_0,W_0)$ and
\[
\lim \frac{1}{\e^2} J_\e(y_\e) = \int_\W \calQ(x,e(u_0(x)))\,dx - \int_{\pl \W} f\cdot R_0 u_0 \, \,d\calH^{n-1} -  \int_\W g\cdot R_0 u_0 \, dx - \frac{1}{2}F(R_0W_0^2).
\]
\end{enumerate}
\end{theorem}

\begin{theorem}[Convergence of minimizers]\label{thm:conv_min}
Let $y_\e \in W^{1,2}(\W;\R^n)$ be a sequence such that
\begin{equation}\label{alm-min}
J_\e(y_\e)\leq \inf_{W^{1,2}} J_\e +o(\e^2).
\end{equation}
Then there exist a sequence $R_\e\in \SO{n}$ and constants $c_\e\in \R^n$ such that, up to subsequences, the rescaled displacements
$$
u_\e(x) = \frac{1}{\e} R_\e^T\brk{y_\e(x) -(R_\e x + c_\e)}
$$
converge to $u_0$ strongly in $W^{1,q}(\W;\R^n)$ for every $1\leq q<2$, $R_\e$ converge to $R_0\in\calR$, and $\frac{1}{\sqrt{\e}}\brk{R_\e - \calP(R_\e)} \to 0$.
Furthermore, $(u_0, R_0)$ is a minimizer of the functional
	\[
	J(u,R) := \int_\W \calQ(x,e(u(x)))\,dx - \int_{\pl \W} f\cdot R u \, \,d\calH^{n-1} -  \int_\W g\cdot R u \, dx
	\]
on $W^{1,2}(\W;\R^n)\times\calR$.
\end{theorem}

\begin{remark}
The results of \cite{MPT19} are an immediate consequence of Theorem~\ref{thm:general_compactness}. 
Indeed, let $y_\e \in W^{1,2}(\W;\R^n)$ be such that $J_\e(y_\e) \le C\e^2$ and let $v_\e=\frac1{\e}(y_\e-\id)$ be the displacement as defined in
\cite{MPT19}.
By \eqref{eq:displacement} we have that
\begin{equation}\label{veue}
\nabla v_\e=R_\e\nabla u_\e +\frac{R_\e-I}{\e}.
\end{equation}
From this relation it is clear that in general one cannot expect $v_\e$ to be bounded in $W^{1,2}$, since the limit $R_0$ of $R_\e$ may be different from $I$ and, even if $R_0=I$, the distance of $R_\e$ from $\calR$ is only of order $\sqrt{\e}$. Assume now that $\calR=\{I\}$. By Theorem~\ref{thm:general_compactness} and equation \eqref{veue} we deduce that $\sqrt{\e}\nabla v_\e$ converge, up to subsequences, to $W_0$ strongly in $L^2$. Moreover, writing $R_\e=e^{\sqrt{\e}W_\e}$, with
$W_\e$ a bounded sequence (Theorem~\ref{thm:general_compactness}),
we obtain
$$
e(v_\e)={\rm sym}(R_\e\nabla u_\e)+{\rm sym}\frac{R_\e-I}{\e}={\rm sym}(R_\e\nabla u_\e)+\frac12 W_\e^2+o(1),
$$
hence $e(v_\e)$ converge, up to subsequences, to $e(u_0)$ weakly in $L^2$, and $e(v_0)=e(u_0)+\frac12 W_0^2$. Thus, we recover the result of \cite{MPT19}.
\end{remark}
\begin{proof1}{of Theorem~\ref{thm:general_compactness}:}
\paragraph{Convergence of $u_\e$ and $R_\e$.} 
By Lemma~\ref{lem:using_FJM}, we have that $u_\e$ is bounded in $W^{1,2}$, from which the first assertion follows.
$\SO{n}$ is compact, hence, by moving to a subsequence, we have $R_\e \to R_0\in \SO{n}$.
Note that the boundedness of $u_\e$ implies that for some $C>0$ we have
\beq\label{eq:aux}
\begin{split}
\frac{1}{\e^2} J_\e(y_\e) 
	&= \frac{1}{\e^2}I_\e(y_\e) - \int_{\pl \W} f\cdot R_\e u_\e \, \,d\calH^{n-1} -  \int_\W g\cdot R_\e u_\e \, dx + \frac{1}{\e}F(I-R_\e) \\
	&\ge -C + \frac{1}{\e}F(I - R_\e).
\end{split}
\eeq
If $R_0\notin \calR$, then $\dist(R_0,\calR)\ge c$ for some constant $c>0$, and since, from the definition of~$\calR$,
\[
\min \{F(I-R) \,\,:\,\, R\in \SO{n},\, \dist(R,\calR)\ge c\} > 0,
\]
we obtain from \eqref{eq:aux} that $\e^{-2} J_\e(y_\e)\to \infty$, in contradiction.
This proves the second assertion.

\paragraph{Convergence of $\e^{-1/2}(R_\e-\calP(R_\e))$.}
First, note that $R_\e \to R_0\in \calR$ implies that $\calP(R_\e)$ is well-defined for small enough $\e$.
We first show that $\dist_{\SO{n}}(R_\e,\calR) = O(\sqrt{\e})$.

To simplify the notation, denote $Q_\e = \calP(R_\e)$ and $d_\e = \dist_{\SO{n}}(R_\e,\calR)$.
We therefore have $R_\e = Q_\e e^{d_\e W_\e}$ for some $W_\e \in N\calR_{Q_\e}$, with $|W_\e| = 1$.
Since  $R_\e \to R_0$, we also have $Q_\e \to R_0$, and therefore, by moving to a subsequence, we have that $W_\e\to W$, where $|W|= 1$ and $W\in N\calR_{R_0}$.
From \eqref{eq:aux} and \eqref{eq:forces_EL_new} we have that for some constant $C>0$,
\[
C \ge -\frac{1}{\e}F(R_\e - Q_\e) = -\frac{1}{\e} \brk{ \frac{d_\e^2}{2} F(Q_\e W_\e^2) + O(d_\e^3)},
\]
we therefore obtain that if $d_\e \gg \sqrt{\e}$, then $F(R_0 W^2) = \lim F(Q_\e W_\e^2) = 0$.
But this is a contradiction since $W$ is a non-zero element of $N\calR_{R_0}$.
We therefore obtain that $d_\e = O(\sqrt{\e})$ as needed.
By moving to a subsequence we have that $d_\e/\sqrt{\e}\to \alpha$ for some $\alpha\ge 0$.

Putting this all together we have
\[
\frac{1}{\sqrt{\e}}\brk{R_\e - Q_\e} = \frac{1}{\sqrt{\e}}(d_\e Q_\e W_\e + O(d_\e^2)) \to \alpha R_0 W, 
\]
which completes the proof as $W_0 = \alpha W \in N\calR_{R_0}$.

\paragraph{Uniqueness of $R_0$ and $e(u_0)$.}
We now show that $R_0$ is independent of the choice of $R_\e$, and that $u_0$ is also independent up to a change by a linear function $Ax + b$, with $A\in \mathbb{M}^{n\times n}_{\text{skew}}$.

Indeed, assume that $R_\e'$ is an alternative choice of rotations, $u_\e'$ are the associated displacements, and let $u_0'$ be their limit.
From Lemma~\ref{lem:using_FJM}, we know that $|R_\e-R_\e'|<C\e$ for some $C>0$; thus, $\lim R_\e' = \lim R_\e = R_0$.

Moreover, writing $R_\e' = R_\e e^{\e A_\e}$ for some uniformly bounded matrices $A_\e \in \mathbb{M}^{n\times n}_{\text{skew}}$, 
we have
\[
\nabla u_\e' = \frac{1}{\e} \brk{(R_\e')^T \nabla y_\e - I} = \frac{1}{\e} \brk{e^{-\e A_\e} R_\e^T \nabla y_\e - I} = \nabla u_\e - A_\e R_\e^T \nabla y_\e + O(\e).
\]
Here $O(\e)$ is with respect to the $L^2$ norm.
By passing to the limit, using the fact that $A_\e$ is antisymmetric and $R_\e^T\nabla y_\e \to I$ strongly in $L^2$ (Lemma~\ref{lem:using_FJM}), we obtain that $u_0' = u_0 + Ax + b$, where $A\in \mathbb{M}^{n\times n}_{\text{skew}}$.

\paragraph{Uniqueness of $W_0$.}
It remains to show that $W_0$ is independent of the choice of $R_\e$.
Assume we have an alternative choice of rotations $R_\e'$.
From Lemma~\ref{lem:using_FJM}, we have that $|R_\e-R_\e'|=O(\e)$.

Denote $Q_\e = \calP(R_\e)$, $Q_\e' = \calP(R_\e')$ and define
\[
d_\e := \dist_{\SO{n}}(R_\e,Q_\e), \qquad d_\e' := \dist_{\SO{n}}(R_\e' , Q_\e').
\]
We have already established the bounds 
\[
\dist_{\SO{n}}(R_\e ,R_\e') = O(\e), \qquad d_\e,\, d_\e' = O(\sqrt{\e}).
\]
From the definition of $Q_\e$ and $Q_\e'$ it therefore follows that
\beq\label{eq:d_e_d_e_prime}
|d_\e - d_\e'| = O(\e), \qquad \dist_{\SO{n}}(Q_\e , Q_\e') = O(d_\e, \e).
\eeq
Indeed, this follows from
\[
d_\e = \dist_{\SO{n}}(R_\e,\calR) \le \dist_{\SO{n}}(R_\e,Q_\e') \le \dist_{\SO{n}}(R_\e,R_\e') +\dist_{\SO{n}}(R_\e',Q_\e') = d_\e' + O(\e),
\]
and similarly when reversing the roles of $R_\e$ and $R_\e'$.

Our goal is to obtain $|Q_\e - Q_\e'| \ll \sqrt{\e}$, which would imply the uniqueness of $W_0$.
If $d_\e \ll \sqrt{\e}$, then we are done by \eqref{eq:d_e_d_e_prime}, since the extrinsic and intrinsic distances on $\SO{n}$ are equivalent.
We can therefore assume that $d_\e \approx \sqrt{\e}$.
Let us write
\[
R_\e = Q_\e e^{d_\e W_\e}, \qquad   Q_\e' = Q_\e e^{t_\e \bar{W}_\e}
\]
where $W_\e, \bar{W}_\e \in \mathbb{M}^{n\times n}_{\text{skew}}$ are of norm $1$, and $t_\e = |Q_\e - Q_\e'| + O(\e)$ (see \eqref{eq:equiv_dist}).
In particular $t_\e\to 0$.

Since both $Q_\e,Q_\e'\in \calR$ are optimal rotations, we obtain from Lemma~\ref{lem:calR_geodesics} that for $\e$ small enough, $Q_\e e^{t \bar{W}_\e}\in \calR$ for any $t\in \R$.
We therefore have, for any $t\in \R$,
\[
d_\e = \dist_{\SO{n}}(R_\e,\calR) \le \dist_{\SO{n}}(R_\e,Q_\e e^{t \bar{W}_\e}) = |d_\e W_\e - t \bar{W}_\e| + O(d_\e^2, t^2).
\]
Let us restrict ourselves to $|t| \le c d_\e$ for some $c>0$.
Since $d_\e \approx \sqrt{\e}$, we obtain that
\[
1 \le | W_\e - \alpha \bar{W}_\e| + O(\sqrt{\e}) \qquad \forall \alpha \in [-c,c].
\]
Now, since $|W_\e| = |\bar{W}_\e| = 1$, we have
\[
| W_\e - \alpha \bar{W}_\e| = \brk{1- 2\alpha \inner{W_\e,\bar{W}_\e} + \alpha^2}^{1/2} \le 1 + \alpha \inner{W_\e,\bar{W}_\e} + \frac{\alpha^2}{2},
\]
from which we obtain that
\[
| \inner{W_\e,\bar{W}_\e}| = O(\e^{1/4}).
\]

On the other hand we have
\[
\dist_{\SO{n}}(R_\e, Q_\e') \le \dist_{\SO{n}}(R_\e ,R_\e') + \dist_{\SO{n}}(R_\e',\calR) = d_\e' + O(\e) = d_\e + O(\e).
\]
Therefore, using again the fact that $d_\e\approx \sqrt{\e}$, we have
\[
\begin{split}
d_\e &\ge \dist_{\SO{n}}(R_\e ,Q_\e') + O(\e)
		= |d_\e W_\e - t_\e \bar{W}_\e| + O(\e, t_\e^2) \\
	&= \brk{d_\e^2 + t_\e^2 - 2t_\e d_\e \inner{W_\e,\bar{W}_\e}}^{1/2} + O(\e, t_\e^2) \\
	&= \sqrt{d_\e^2 + t_\e^2}\brk{1 - \frac{2t_\e d_\e}{d_\e^2 + t_\e^2} \inner{W_\e,\bar{W}_\e}}^{1/2} + O(\e, t_\e^2) \\
	&\ge \sqrt{d_\e^2 + t_\e^2}\brk{1 - |\inner{W_\e,\bar{W}_\e}|}^{1/2} + O(\e, t_\e^2), \\
\end{split}
\]
which implies that $t_\e \ll d_\e = O(\sqrt{\e})$, hence $|Q_\e - Q_\e'| \ll \sqrt{\e}$, which completes the uniqueness proof.
\end{proof1}

\begin{proof1}{of Theorem~\ref{thm:general_Gamma}:}
\paragraph{Lower bound.}

First consider the elastic part $\e^{-2}I_\e(y_\e)$.
We have, using frame indifference,
\[
I_\e(y_\e) = I_\e(R_\e(x + \e u_\e(x))) = I_\e(x+ \e u_\e(x)) = \int_\W \calW(I+ \e \nabla u_\e(x))\,dx.
\]
Taylor expanding $W(I+A)$, we have from the regularity assumption \ref{item:regularity} and \eqref{eq:Q_symmetric} that
\[
\Abs{\calW(x,I+ \e \nabla u_\e) - \e^2 \calQ(x,e(u_\e))} \le \w(\e |\nabla u_\e|),
\]
where $\w(t)$ is a non-negative function satisfying $\lim_{t\to0} \w(t)/t^2 = 0$.
We therefore have
\[
\begin{split}
\frac{1}{\e^2} I_\e(y_\e) &\ge \int_\W \brk{ \calQ(x,e(u_\e)) - \frac{\w(\e |\nabla u_\e|)}{\e^2}}\,dx \\
	& \ge \int_\W \chi_\e \brk{ \calQ(x,e(u_\e)) - \frac{\w(\e |\nabla u_\e|)}{\e^2}}\,dx \\
	& = \int_\W \brk{  \calQ(x,\chi_\e^{1/2} e(u_\e)) - \chi_\e |\nabla u_\e|^2 \frac{\w(\e |\nabla u_\e|)}{\e^2 |\nabla u_\e|^2}}\,dx,
\end{split}
\]
where 
\begin{equation}\label{chie}
\chi_\e(x) = \begin{cases} 1 &\text{ if } |\nabla u_\e(x)| < \e^{-1/2}, \\ 0 & \text{ if } |\nabla u_\e(x)|\ge \e^{-1/2}. \end{cases}
\end{equation}
Since $u_\e\weakly{} u_0$ in $W^{1,2}$, we have that $\chi_\e \to 1$ in $L^2$ and therefore also $\chi_\e^{1/2} e(u_\e) \weakly{} e(u_0)$ in $L^2$.
Therefore, since $\calQ(x,\cdot)$ is positive-semidefinite (and in particular, convex), we have that
\[
\liminf \int_\W  \calQ(x,\chi_\e^{1/2} e(u_\e))\,dx \ge \int_\W \calQ(x,e(u_0))\,dx.
\]
From this, and the fact that $\chi_\e\frac{\w(\e |\nabla u_\e|)}{\e^2 |\nabla u_\e|^2}\to 0$ uniformly, we obtain that
\begin{equation}\label{liminf_Ie}
\liminf \frac{1}{\e^2} I_\e(y_\e) \ge \int_\W \calQ(x,e(u_0))\,dx.
\end{equation}
Next, since $R_\e u_\e \weakly{} R_0 u_0$ in $W^{1,2}$, we have that
\[
\int_{\pl \W} f\cdot R_\e u_\e \, \,d\calH^{n-1} +  \int_\W g\cdot R_\e u_\e \, dx \to 
\int_{\pl \W} f\cdot R_0 u_0 \, \,d\calH^{n-1} +  \int_\W g\cdot R_0 u_0 \, dx.
\]
Finally, writing $R_\e = \calP(R_\e) e^{\sqrt{\e}W_\e}$, we have that
\[
\frac{1}{\e} F(R_\e - I) = \frac{1}{\e} F(R_\e - \calP(R_\e)) = \frac{1}{2} F\brk{\calP(R_\e) W_\e^2} + O(\e) \to \frac{1}{2}F(R_0 W_0^2).
\]
Putting all these together, we have
\[
\begin{split}
&\liminf_{\e\to 0} \frac{1}{\e^2} J_\e(y_\e) \\
	&\quad = \liminf_{\e\to 0} \frac{1}{\e^2}I_\e(y_\e) - \lim_{\e\to 0}\brk{\int_{\pl \W} f\cdot R_\e u_\e \, \,d\calH^{n-1} +  \int_\W g\cdot R_\e u_\e \, dx} - \lim_{\e\to 0} \frac{1}{\e} F(R_\e - I) \\
	&\quad \ge \int_\W \calQ(x, e(u_0))\,dx - \int_{\pl \W} f\cdot R_0 u_0 \, \,d\calH^{n-1} -  \int_\W g\cdot R_0 u_0 \, dx - \frac{1}{2}F(R_0 W_0^2),
\end{split}
\]
which completes the proof of the lower bound.

\paragraph{Upper bound.}
For $u_0 \in W^{1,2}$, choose a sequence $u_\e \in W^{1,\infty}$ such that $u_\e\to u_0$ in $W^{1,2}$ and $\|\nabla u_\e\|_{\infty} < \e^{-1/2}$. 
Define $y_\e := R_0 e^{\sqrt{\e}W_0}  (x + \e u_\e)$.
In this case we have $R_\e = R_0 e^{\sqrt{\e}W_0}$ and $u_\e$ is indeed the displacement of $y_\e$ as in \eqref{eq:displacement}.
Note that since $R_0\in \calR$ and $W_0 \in N\calR_{R_0}$, we have that $R_0 = \calP(R_\e)$.
It follows that $y_\e \to (u_0,R_0,W_0)$ as needed.
Now, similarly as in the lower bound, we have
\[
\begin{split}
\Abs{\frac{1}{\e^2} I_\e(y_\e) - \int_\W \calQ(x, e(u_\e))\,dx} &\le \int_\W \frac{\w(\e |\nabla u_\e|)}{\e^2}\,dx
	\le  \int_\W |\nabla u_\e|^2 \frac{\w(\e |\nabla u_\e|)}{\e^2|\nabla u_\e|^2}\,dx \to 0,
\end{split}
\]
since $\e \|\nabla u_\e\|_{\infty} = O(\sqrt{\e})$.
Now, since $u_\e \to u_0$ strongly in $W^{1,2}$ and $D_A^2\calW(\cdot,I)$ is in $L^\infty$, we have that $\int_\W \calQ(x, e(u_\e))\,dx \to \int_\W \calQ(x, e(u_0))\,dx$.

The forces part behaves exactly as in the lower bound, yielding
\[
\begin{split}
\lim_{\e\to 0} \frac{1}{\e^2} J_\e(y_\e)
	= \int_\W \calQ(x, e(u_0))\,dx - \int_{\pl \W} f\cdot R_0 u_0 \, \,d\calH^{n-1} -  \int_\W g\cdot R_0 u_0 \, dx - \frac{1}{2}F(R_0 W_0^2).
\end{split}
\]
\end{proof1}

\begin{proof1}{of Theorem~\ref{thm:conv_min}:}
By Proposition~\ref{prop:inf_energy} we have that $J_\e(y_\e) < C\e^2$, hence by Theorem~\ref{thm:general_compactness} there exist 
$u_0\in W^{1,2}(\W;\R^n)$, $R_0\in\calR$, and $W_0 \in N\calR_{R_0}$ such that $u_\e \weakly{} u_0$ in $W^{1,2}$,
$R_\e \to R_0$, and
\beq\label{liminf-min}
\liminf \frac{1}{\e^2} J_\e(y_\e) \ge J(u_0,R_0) - \frac{1}{2}F(R_0W_0^2),
\eeq
where we used the lower bound in Theorem~\ref{thm:general_Gamma}. 

Let now $v\in W^{1,2}$ and $R\in\calR$. By the upper bound in Theorem~\ref{thm:general_Gamma} with $W_0=0$ there exists a sequence
$v_\e\in W^{1,2}$ such that 
\begin{equation}\label{limsup-min}
\lim \frac{1}{\e^2} J_\e(v_\e) = J(v,R) .
\end{equation}
Combining \eqref{alm-min}, \eqref{liminf-min}, and \eqref{limsup-min}, we deduce
\begin{equation}\label{min-ineq}\begin{split}
J(u_0,R_0) - \frac{1}{2}F(R_0W_0^2) & \leq \liminf \frac{1}{\e^2} J_\e(y_\e)\leq \limsup \frac{1}{\e^2} J_\e(y_\e) = \limsup \inf_{W^{1,2}}
\frac{1}{\e^2} J_\e \\
& \leq \lim \frac{1}{\e^2} J_\e(v_\e)= J(v,R) .
\end{split}\end{equation}
Therefore, $(u_0, R_0)$ is a minimizer of the functional $J$ on $W^{1,2}\times\calR$, and $W_0 = 0$ (this follows from \eqref{eq:forces_EL_new} and the definition of $N\calR_{R_0}$).

It remains to show that $u_\e$ converge to $u_0$ strongly in $W^{1,q}$ for every $1\leq q<2$. Choosing $v=u_0$ and $R=R_0$ in \eqref{min-ineq} we obtain
$$
 \lim \frac{1}{\e^2} J_\e(y_\e)=J(u_0,R_0),
$$
hence
$$
\frac{1}{\e^2} I_\e(y_\e) - \frac{1}{\e} F(R_\e - I) \to \int_\W \calQ(x, e(u_0))\,dx.
$$
Equation \eqref{liminf_Ie} and the fact that $I$ is an optimal rotation imply that $\frac{1}{\e} F(R_\e - I)\to0$ and
\begin{equation}\label{energy-eq}
 \lim \frac{1}{\e^2} I_\e(y_\e)=\int_\W \calQ(x, e(u_0))\,dx.
\end{equation}
Let now $\chi_\e$ be defined as in \eqref{chie}. From the proof of the lower bound in Theorem~\ref{thm:general_Gamma} it follows that
\[
\begin{split}
 \lim \frac{1}{\e^2} I_\e(y_\e) & \geq \limsup \int_\W  \calQ(x,\chi_\e^{1/2} e(u_\e))\,dx  \\
& \ge \liminf \int_\W  \calQ(x,\chi_\e^{1/2} e(u_\e))\,dx \ge \int_\W \calQ(x,e(u_0))\,dx. \end{split}
\]
Therefore, by \eqref{energy-eq} we obtain
\begin{equation}\label{Q-cont}
 \lim \int_\W  \calQ(x,\chi_\e^{1/2} e(u_\e))\,dx = \int_\W \calQ(x,e(u_0))\,dx.
\end{equation}
By the coercivity of $\calQ$ we have that
\[
\begin{split}
c\int_\W |\chi_\e^{1/2} e(u_\e)-e(u_0)|^2\, dx  & \leq \int_\W  \calQ(x,\chi_\e^{1/2} e(u_\e)-e(u_0))\,dx \\
& = \int_\W  \calQ(x,\chi_\e^{1/2} e(u_\e))\,dx -\int_\W  D_A^2\calW(x,I)(\chi_\e^{1/2} e(u_\e),e(u_0))\,dx \\
& \quad+ \int_\W  \calQ(x,e(u_0))\,dx.
\end{split}
\] 
We now use the weak convergence of $\chi_\e^{1/2} e(u_\e)$ to $e(u_0)$ in $L^2$, the boundedness of $D_A^2\calW(x,I)$, and equation \eqref{Q-cont}, to deduce that
$\chi_\e^{1/2} e(u_\e)\to e(u_0)$ strongly in $L^2$.
Since $\chi_\e\to 1$ in $L^p$ for every $1\leq p<\infty$ and $e(u_\e)$ is bounded in $L^2$, we have that
$(1-\chi_\e^{1/2})e(u_\e)\to 0$ strongly in $L^q$ for every $1\leq q<2$, hence $e(u_\e)\to e(u_0)$ strongly in $L^q$ for every $1\leq q<2$.
By Korn's inequality there exists, for every $q\in (1,2)$, a constant $c_q$ such that
$$
\int_\W |\nabla u_\e-\nabla u_0|^q\, dx\leq c_q\int_\W |e(u_\e)-e(u_0)|^q\, dx + c_q\int_\W |u_\e-u_0|^q\, dx.
$$
By the Rellich Theorem $u_\e\to u_0$ strongly in $L^q$, hence we conclude that $u_\e\to u_0$ strongly in $W^{1,q}$ for every $q\in (1,2)$. 
The convergence for $q=1$ follows immediately since $\W$ is a bounded domain.
\end{proof1}

\section{Classification and examples of optimal rotations}\label{sec:examples}

In this section we classify the possible sets $\calR$ of optimal rotations, in dimensions $n=2,3$.
The optimal rotations are derived from the functional $F\in (\mathbb{M}^{n\times n})^*$.
Endowing $\mathbb{M}^{n\times n}$ with the Frobenius inner-product, we can identify $F$ with an $n\times n$ matrix, which we will also denote by $F$;
since $F(W) = 0$ for any $W\in \mathbb{M}^{n\times n}_{\text{skew}}$, it follows that $F$ is a symmetric matrix.
Note that the assumption $I\in \calR$ gives further restrictions on $F$, as seen in \eqref{eq:FDi}; in particular, it cannot be an arbitrary symmetric matrix.

\begin{proposition}[Classification of optimal rotations in 2D]
When $n=2$, the set of optimal rotations is either $\calR=\{I\}$ or $\calR = \SO{2}$.
The latter case happens if and only if $\tr F = 0$.
\end{proposition}

\begin{proof}
Since $\calR$ is a complete, connected, closed, boundryless submanifold of $\SO{2}$, and $\SO{2}$ is one dimensional, $\calR$ is either a singleton or the whole $\SO{2}$.
Since $R\in \SO{2}$ implies that $-R\in \SO{2}$, the case $\mathcal{R} = \SO{2}$ happens if and only if $F(R)=0$ for every $R\in \SO{2}$.
Since $F$ is symmetric and $R= \brk{\begin{matrix} \cos \alpha & -\sin \alpha \\ \sin \alpha & \cos \alpha \end{matrix} }$ for some angle $\alpha$, this holds if and only if $F$ is traceless.
\end{proof}

\begin{proposition}[Classification  of optimal rotations in 3D]
When $n=3$, the set of optimal rotation is either $\calR=\{I\}$ or one of the following:
\begin{itemize}
\item $\calR = \SO{3}$, if and only if $F\equiv 0$.
\item $\calR$ is isometric to the real projective plane $\mathbb{P}_2(\R) \cong S^2\big/\sim$, where $\sim$ is the identification of antipodal points and $S^2$ is the round sphere.
	This happens if and only if the eigenvalues of $F$ are $a,a,-a$ for some $a>0$.
\item $\calR$ is a single closed geodesic (that is, it is isometric to $\SO{2} \cong S^1$);
	this happens if and only if the eigenvalues of $F$ are $b,a,-a$ for some $b>a\ge 0$.
\end{itemize}
\end{proposition}

\begin{proof}
\paragraph{Classification of the possible isometry classes of $\calR$.}
Assume that $\calR\neq\{I\}$, hence it is a closed, connected, boundryless totally-geodesic submanifold of $\SO{3}$.
In particular, $\calR$ is the image of the exponential map of $\SO{3}$, restricted to the subspace $T\calR_I\subset T\SO{3}_I$. This is because every complete manifold is the image of its exponential map, and the exponential map of a totally-geodesic submanifold is the exponential map of the ambient manifold restricted to the tangent plane of the submanifold.
It follows that if $\dim \calR = 1$, then $\calR$ consists of a single, closed geodesic.
If $\dim \calR = 3 = \dim \SO{3}$, then $T\calR_I = T\SO{3}_I$, hence $\calR = \SO{3}$.

Note that $\SO{3}$, with the metric induced from $\mathbb{M}^{3\times 3}$, is isometric to $S^3/{\sim}$, where $S^3$ is the round 3-sphere, and $\sim$ is the identification of antipodal points. 
This follows since in both cases the metric obtained is bi-invariant with respect to the group action, and such a metric is unique.\footnote{In the case of $S^3$, with its canonical embedding into $\R^4$, the group action is quaternion conjugation, where we identify $p=(p_1,p_2,p_3,p_4)\in S^3$ with the quaternion $p_1 + p_2{\bf i} +p_3 {\bf j} + p_4 {\bf k}$.}
Denote by $\pi : S^3 \to \SO{3}$ the covering map. 
If $\dim \calR = 2$, then $\pi^{-1} \calR$ is a connected, totally-geodesic, complete two-dimensional submanifold of $S^3$, hence it is isometric to the round $S^2$ (since the image of a two-dimensional subspace of $T S^3_p$ under the exponential map of $S^3$ is isometric to $S^2$).
Thus $\calR$ is isometric to $\mathbb{P}_2(\R) = S^2\big/ \sim$.
This completes the classification of the possible isometry classes of $\calR$.

\paragraph{The principal curvatures of $\SO{n}$ in $\mathbb{M}^{n\times n}$.}
In order to relate the eigenvalues of $F$ to the structure of $\calR$, we need first to recall the second fundamental form of $\SO{n}$ in $\mathbb{M}^{n\times n}$.\footnote{This is by no means a new result; here we follow the presentation as in \cite{Bry18}.}
Generally, the second fundamental form of a submanifold $\mathcal{M} \subset \mathcal{N}$ at $p\in \mathcal{M}$ is the vector-valued quadratic form
$\textup{II}_p: T\mathcal{M}_p \to N\mathcal{M}_p$ defined by $\textup{II}_p(X) := \nabla^{\mathcal{N}}_X X - \nabla^{\mathcal{M}}_X X$ (here $N\mathcal{M}_p$ is the normal bundle of $\mathcal{M}$ at $p$, and $\nabla^{\mathcal{M}}$ and $\nabla^{\mathcal{N}}$ are the Levi-Civita covariant derivatives of $\mathcal{M}$ and $\mathcal{N}$, respectively).
The second fundamental form of $\mathcal{M}$ in direction $\eta \subset N\mathcal{M}_p$ is the quadratic form $X\mapsto \inner{\textup{II}_p(X), \eta}$, and the principal curvatures of $\mathcal{M}$ in direction $\eta$ are the eigenvalues of this form (with respect to an orthonormal basis of $T\mathcal{M}_p$).
If $\mathcal{M}$ is totally geodesic in $\mathcal{N}$, then its second form vanishes identically.

Now let $\mathcal{N} = \R^D$. Since $T\mathcal{M}_p \oplus N\mathcal{M}_p =\R^D$, we can write $\mathcal{M}$, at the vicinity of $p$, as a graph of a function $f:T\mathcal{M}_p \to N\mathcal{M}_p$, whose differential at $p$ vanishes.
In this case we can identify the second fundamental form as the quadratic correction of $f$, that is, $f(X) = f(0) + \textup{II}(X) + O(|X|^3)$.

In our case, the tangent and normal spaces of $\SO{n}$ at $I$ are $\mathbb{M}^{n\times n}_{\text{skew}}$ and $\mathbb{M}^{n\times n}_{\text{sym}}$, respectively.
The map $W\mapsto e^{W}$ maps $\mathbb{M}^{n\times n}_{\text{skew}}$ to $\SO{n}$; the decomposition of $e^W$ into skew and symmetric parts is given by
\[
e^W = \sinh W + \cosh W = \sinh W + \sqrt{I + \sinh^2 W}.
\]
Therefore, since $W\mapsto \sinh W$ is a diffeomorphism of $\mathbb{M}^{n\times n}_{\text{sym}}$ at the vicinity of $0$, we obtain that $\SO{n}$ is the graph of the function $f: \mathbb{M}^{n\times n}_{\text{skew}}\to \mathbb{M}^{n\times n}_{\text{sym}}$, defined by
\[
f(W) = \sqrt{I + W^2} = I + \frac{1}{2} W^2 + O(|W|^4)
\]
for small enough $W$.
Thus the second form of $\SO{n}$ at the identity is $\textup{II}(W) = \frac{1}{2}W^2$.
The second fundamental form in a direction $S \in \mathbb{M}^{n\times n}_{\text{sym}}$ is then the map $W\mapsto \inner{\frac{1}{2}W^2 , S}$.
If $s_1,\ldots, s_n$ are the eigenvalues of $S$, then a direct calculation shows that $-\frac{1}{4}(s_i + s_j)$, $i<j$ are the principal curvatures of $\SO{n}$ at $I$ in direction $S$.\footnote{Indeed, consider, for $1\le i<j\le n$, the orthonormal basis $W_{ij} = \frac{1}{\sqrt 2}(e_{ij} - e_{ji})$ of $\mathbb{M}^{n\times n}_{\text{skew}}$, where $e_{ij}$ is the standard matrix basis.
If $S$ is diagonal with entries $s_1,\ldots,s_n$, then for $W= \sum_{i<j} \alpha_{ij}W_{ij}$, we have that
\[
\inner{\frac{1}{2}W^2 , S} = -\frac{1}{4}\sum_{i<j} \alpha_{ij}^2 (s_i+s_j),
\]
showing that the eigenvalues are $-\frac{1}{4}(s_i + s_j)$.
For a general $S$, we have that $S = R^T D R$ for some rotation $R$ and diagonal matrix $D$. The calculation is then similar, using the orthonormal basis $R^T W_{ij} R$.}

Back to our case, we show that the second form of $\SO{n}$ at the identity in the direction $F$ is negative semi-definite.
That is, if $f_1,\ldots, f_n$ are the eigenvalues of $F$, then $f_i+f_j \ge 0$ for all $i\ne j$.
Assume otherwise, and without loss of generality assume that $f_1+f_2 <0$.
This contradicts \eqref{eq:FDi}: 
indeed, we can write $F = R^T \diag(f_1,\ldots, f_n) R$ for some $R\in \SO{n}$, and then, with the notation of \eqref{eq:FDi}, we obtain
\[
F(R^T D_1 R) = \inner{\diag(f_1,\ldots, f_n) , D_1} = f_1 + f_2 <0,
\]
which is a contradiction to \eqref{eq:FDi}.

\paragraph{The relation between eigenvalues of $F$ and $\dim \calR$.}
Denote by $H$ the hyperplane $H := F^{-1}\{F(I)\} \subset \mathbb{M}^{3\times 3}$.
The normal to $H$ is, by definition, the matrix $F$.
We have the inclusions
\[
\calR \subset \SO{3} \subset \mathbb{M}^{3\times 3} \quad\text{and}\quad  \calR \subset H \subset \mathbb{M}^{3\times 3}.
\]
In what follows, $\textup{II}^{\calR, H}$ denotes the second fundamental form of $\calR$ in $H$ at $I$, and similarly for the other inclusions; $\textup{II}^{\calR, H}_F$ denotes the second fundamental form in direction $F$ at $I$, and so on.
Since $H$ is a hyperplane, it is totally geodesic in $\mathbb{M}^{3\times 3}$.
It follows that $\textup{II}^{\calR, \mathbb{M}^{3\times 3}}_F$ vanishes:
\[
\begin{split}
\textup{II}^{\calR, \mathbb{M}^{3\times 3}}_F(W) := \inner{\textup{II}^{\calR, \mathbb{M}^{3\times 3}}(W) ,F}
	&= \inner{\nabla^{\mathbb{M}^{3\times 3}}_{W} W - \nabla^\calR_W W, F} \\
	&= \inner{\nabla^{\mathbb{M}^{3\times 3}}_{W} W - \nabla^{H}_W W , F} + \inner{\nabla^{H}_W W - \nabla^\calR_W W, F} \\
	&= \inner{\nabla^{H}_W W - \nabla^\calR_W W, F},
\end{split}
\]
where we used the fact that $H$ is totally geodesic in $\mathbb{M}^{3\times 3}$ and thus $\nabla^{\mathbb{M}^{3\times 3}}_{W} W = \nabla^{H}_W W$.
Now, since $\calR\subset H$, $\nabla^{H}_W W - \nabla^\calR_W W$ is a tangent vector to $H$; on the other hand, $F$ is perpendicular to $H$, hence $\textup{II}^{\calR, \mathbb{M}^{3\times 3}}_F=0$.
On the other hand,
since $\calR\subset \SO{3}$ is totally geodesic, $\textup{II}^{\calR, \SO{3}} = 0$. 
Thus, by a similar argument (with $\SO{3}$ instead of $H$ and without taking the inner product with $F$), we obtain that $\textup{II}^{\calR, \mathbb{M}^{3\times 3}} = \left.\textup{II}^{\SO{3}, \mathbb{M}^{3\times 3}} \right|_{T\calR_I}$.
Thus we obtain that
\[
\left.\textup{II}^{\SO{3}, \mathbb{M}^{3\times 3}}_F \right|_{T\calR_I} \equiv 0.
\]
Recall that $\textup{II}^{\SO{3}, \mathbb{M}^{3\times 3}}_F$ is a negative semi-definite quadratic form.
Since it vanishes on a subspace of dimension $\dim \calR$, it follows that at least $\dim \calR$ of the principal curvatures of $\SO{n}$ in the direction $F$ vanish.
As shown above, the principal curvatures are $-\frac{1}{4}(f_1+f_2)$, $-\frac{1}{4}(f_2+f_3)$ and $-\frac{1}{4}(f_1+f_3)$, where $f_i$ are the eigenvalues of $F$.
\begin{itemize}
\item If $\dim \calR = 3$, it follows that $f_1=f_2=f_3=0$, and thus $F=0$.
	Obviously, if $F=0$ then $\calR = \SO{3}$ and thus $\dim \calR = 3$.
\item If $\dim \calR = 2$, we have that, without loss of generality $f_1=f_2 =-f_3$.
	Since $\textup{II}^{\SO{3}, \mathbb{M}^{3\times 3}}_F$ is negative semi-definite, we have that $f_1+f_2 \ge 0$; if equality holds, then $F=0$ and $\dim \calR = 3$.
	We thus obtain that $\dim \calR = 2$ implies that the eigenvalues of $F$ are $a,a,-a$ for some $a>0$.
\item If $\dim \calR = 1$, we have that, without loss of generality, $f_2 = -f_3$.
	Again, the negative semi-definiteness of $\textup{II}^{\SO{3}, \mathbb{M}^{3\times 3}}_F$ implies that $f_1 \ge |f_2| = |f_3|$;
	thus  $\dim \calR = 1$ implies that the eigenvalues of $F$ are $b,a,-a$ for some $b>a\ge 0$.
\end{itemize}

In order to complete the proof we need to show that if the eigenvalues of $F$ are $a,a,-a$ for some $a>0$ then $\dim \calR = 2$, and if they are $b,a,-a$ for $b> a \ge 0$, then $\dim \calR = 1$.
Assume that for some $Q\in \SO{3}$,
\[
F = Q^T \diag(a,a,-a)\, Q.
\]
Thus, for a general matrix $R\in \SO{3}$, we have that
\[
F(Q^T R Q) = a(R_{11} + R_{22} - R_{33}).
\]
Writing $R$ in a quaternion representation, that is $R = p_1 + p_2{\bf i} +p_3 {\bf j} + p_4 {\bf k}$ for a unit vector $p=(p_1, p_2, p_3, p_4)$, we obtain that 
\[
F(Q^T R Q) = a(1 - 4 p_4^2).
\]
Thus $\calR$ is the two-dimensional submanifold $Q\{p_4=0\}Q^T$.

Next, assume that for some $Q\in \SO{3}$ and $b>a\ge 0$, we have
\[
F = Q^T \diag(b,a,-a)\, Q.
\]
In this case $F(Q^T R Q)$ is maximized for all rotations $R$ around the $x$-axis.
Thus $\dim\calR \ge 1$, and since $b>a$, we have that $\dim\calR = 1$.
\end{proof}

\begin{example}[Uniform tension]
Let $\W \subset \R^n$ be a Lipschitz domain, and denote by $\nu$ the outer normal of $\pl \W$.
Let the traction force $f$ be $f=\nu$, and set the body force $g$ to be zero.
We then have, using the divergence theorem, that
\[
F(A) := \int_{\pl\W} Ax \cdot \nu \,d\calH^{n-1} = |\W| \tr(A).
\]
It immediately follows that $I$ is the unique maximizer of $F$ on $\SO{n}$.
That is, $\calR = \{I\}$ in this case.\footnote{This example essentially appears in \cite[Remark~2.8]{MPT19}.}
\end{example}

\begin{example}[Uniform compression]
Reversing the sign from the previous example, that is, taking $f = -\nu$, we obtain
\[
F(A) = -|\W| \tr(A).
\]
In this case $I$ is a minimizer of $F$ among rotations, hence, in order to use the formalism of this paper, we first need to rotate the system by a maximizer of $F$.\footnote{Compare with \cite[Remark~2.7, Example~4.6]{MPT19}.}

If $n=2$ (or more generally, if $n$ is even), then $-I$ is a maximizer, and rotating by it reduces this example to the previous one, with a unique maximizer.

If $n=3$, we recall that for $R = p_1 + p_2{\bf i} +p_3 {\bf j} + p_4 {\bf k}$, $\tr(R) = 3 - 4(p_2^2 + p_3^2 + p_4^2)$.
Thus, a maximizer of $F$ in $\SO{3}$ is any rotation with $p_1=0$ (that is, a rotation by $\pi$ around any axis).
In particular, we obtain that $\calR$ is two-dimensional in this case.
\end{example}

\begin{example}[Tangential forces]
Consider now the two dimensional case $n=2$, and let the traction force be $f=Z\tau$, where $\tau$ is the unit tangent to $\pl \W$, and $Z$ is a reflection matrix, say, a reflection by the $x_2$ axis.
If there are no body forces, we have (by Green's theorem),
\[
F(A) := \int_{\pl\W} ZAx \cdot \tau \,d\calH^{1} = |\W| (A_{12} + A_{21}).
\]
In particular, $F|_{\SO{2}} = 0$, and thus $\calR = \SO{2}$.
By considering a cylinder $\W\times (0,1)$, this example can be lifted to three dimensions, thus obtaining a three-dimensional example in which $\dim \calR = 1$.
\end{example}

\begin{example}[Full degeneracy]
In dimensions $n>2$, $\calR = \SO{n}$ implies that $F\equiv 0$ (the previous example is a counterexample for this for $n=2$).
However, as the following example shows, $F \equiv 0$ does not imply that the forces themselves must be zero.
Let $\W$ be the unit ball, and consider zero traction forces $f\equiv 0$ and a body force $g(x) = \rho(|x|)e_1$ for some sufficiently nice function $\rho:(0,1)\to \R$.
In order for the forces to be equilibrated \eqref{eq:total_force_2}, we must have
\[
0 = \int_{\W} \rho(|x|)\,dx = n\omega_n \int_0^1 \rho(r) \,r^{n-1} \,dr,
\]
where $\omega_n$ is the measure of the unit ball.
For example, if $n=3$, we can take $\rho(r) = 1 -\frac{4}{3}r$ or $\rho(r) = \frac{1}{r^2} - \frac{2}{r}$.
For any such force, we obtain that $F\equiv 0$:
\[
F(A) = \sum_{j=1}^n A_{1j} \int_\W \rho(r) x_j \,dx = 0,
\]
since the domain is a ball.
\end{example}

\begin{example}[Gravity field]
In dimension $n=3$ let the traction force $f$ be zero and let the body force $g$ be given by the gravity field
$$
g(x)=-\bar g\bar\rho(x)e_3 \qquad \text{ with } \ \bar\rho(x):= \rho(x)-\frac1{|\Omega|}\int_\Omega \rho(z)\,dz,
$$
where $\bar g$ is the gravitational constant and $\rho\in L^2(\Omega)$ is the mass density. The normalization constant $-\frac1{|\Omega|}\int_\Omega \rho(z)\,dz$ is introduced to guarantee the forces to be equilibrated. By direct computations we have
$$
F(A) = -\bar g \sum_{j=1}^3A_{3j} \int_\Omega \bar\rho(x) x_j\,dx.
$$
Set $b:= \int_\Omega \bar\rho(x) x\,dx$. If $b=0$, then $F\equiv 0$ and $\calR = \SO{3}$. If $b\neq0$, then $\mathcal R$ is the set of all rotations having $-b/|b|$ as third row. Note that this is a mechanically relevant example, which is covered by our analysis (after rotating the system, so that $I\in\calR$), whereas the compatibility assumption of \cite{MPT19} is not satisfied. 
\end{example}

\paragraph{Acknowledgements}
CM was partially supported by ISF grant 1269/19. MGM acknowledges support by the Universit\`a degli Studi di Pavia through the 2017 Blue Sky Research Project ``Plasticity at different scales: micro to macro", by MIUR--PRIN 2017, and by GNAMPA--INdAM.

\appendix
\section{An example for Lemma~\ref{lem:calR_geodesics}}\label{app:example}

Here we show that, for $n>3$, Lemma~\ref{lem:calR_geodesics} does not imply that if $R_0$ and $R_1$ are two distinct elements of $\calR$,
then \emph{any} geodesic between $R_0$ and $R_1$ lies in $\calR$.
Let
\[
S := \brk{\begin{matrix} 
0 &  &  & \\ 
 & 0 & & \\
 &  & 1 & \\
 &  &  & 1
\end{matrix} }, \qquad
F(A) := \inner{S,A}.
\]
Since all the entries of a rotation matrix are between $-1$ and $1$, it is obvious that $R_0 :=I\in \calR$.
Choose $\lambda$ and $\mu$ such that $\rho := \lambda/\mu$ is not an integer, and let
\[
W_0 = 
\brk{\begin{matrix} 
0 & \lambda &  & \\ 
-\lambda & 0 & & \\
 &  & 0 & \mu\\
 &  & -\mu & 0
\end{matrix} }.
\]
We then have
\[
F(e^{tW_0}) = 2\cos(\mu t),
\]
hence $e^{tW_0}\in \calR$ if and only if $t\in \frac{2\pi}{\mu} \bbZ$, and since $\lambda/\mu$ is not an integer, $R_1 := e^{\frac{2\pi}{\mu} W_0}\ne I$.
In other words, the geodesic $e^{tW_0}$ between $I$ and $R_1$ does not belong to $\calR$.
The geodesic connecting $I$ and $R_1$ that does belongs to $\calR$ is $e^{tW_1}$, where 
\[
W_1:=
\brk{\begin{matrix} 
0 & \lambda &  & \\ 
-\lambda & 0 & & \\
 &  & 0 & 0\\
 &  & 0 & 0
\end{matrix} }.
\]

In dimensions $n=2,3$ this cannot happen. In these dimensions we have the Rodrigues formula
\beq\label{eq:Rodrigues}
\exp (tW) = I + \sin t W + (1-\cos t)W^2,
\eeq
whenever $W\in\mathbb{M}^{n\times n}_{\text{skew}}$, $|W| =\sqrt{2}$.\footnote{The Rodrigues formula \eqref{eq:Rodrigues} easily follows from \eqref{eq:skew_canonical_form}, since in dimensions $n=2,3$ we have $k=1$ in \eqref{eq:skew_canonical_form}, and $|W|=\sqrt{2}$ then implies $\lambda_1 = \pm 1$, from which it follows that $W^3=-W$.}
Let $R_0,R_1\in\calR$. 
If $R_1=R_0 e^{t_0W}$ for some $t_0\ne 0$ and $W\in\mathbb{M}^{n\times n}_{\text{skew}}$, $|W| =\sqrt{2}$, then $F(R_0)=F(R_1)$, together with \eqref{eq:Rodrigues} imply that $F(R_0W^2) = 0$.
Using \eqref{eq:Rodrigues} again (or Lemma~\ref{lem:F_W_squared_zero}), we have that $R_0e^{tW}\in \calR$ for every $t\in \R$.
In other words, the assumption in Lemma~\ref{lem:calR_geodesics}, that $W$ needs to be of the form of Lemma~\ref{lem:geodesics_SO}, can be dropped in dimensions $n=2,3$.


{\footnotesize
\bibliographystyle{amsalpha}
\providecommand{\bysame}{\leavevmode\hbox to3em{\hrulefill}\thinspace}
\providecommand{\MR}{\relax\ifhmode\unskip\space\fi MR }
\providecommand{\MRhref}[2]{%
  \href{http://www.ams.org/mathscinet-getitem?mr=#1}{#2}
}
\providecommand{\href}[2]{#2}

}

\Addresses

\end{document}